\documentclass{article}

\usepackage{verbatim}
\usepackage{graphics}
\usepackage[pdftex]{graphicx}
\usepackage{amsmath, amssymb, amsfonts, amsthm}
\usepackage{url}
\usepackage{setspace}
\usepackage{algorithmic}
\usepackage{enumitem}
\usepackage{tabularx}
\setlist{nolistsep}

\def \PFq[#1]{\mathbf{P}^#1(\mathbb{F}_q)}
\def \PF[#1,#2]{PG(#1,#2)}

\newtheorem{theorem}{Theorem}

\newtheorem{lemma}[theorem]{Lemma}
\newtheorem{definition}[theorem]{Definition}
\newtheorem{corollary}[theorem]{Corollary}

\begin{document}

\title{Essential dimension and the flats spanned by a point set}
\author{Ben Lund \footnote{Work on this paper was supported by NSF grant CCF-1350572 and by ERC Grant 267165 DISCONV.}}

\maketitle

\begin{abstract}
Let $P$ be a finite set of points in $\mathbb{R}^d$ or $\mathbb{C}^d$.
We answer a question of Purdy on the conditions under which the number of hyperplanes spanned by $P$ is at least the number of $(d-2)$-flats spanned by $P$.

In answering this question, we define a new measure of the degeneracy of a point set with respect to affine subspaces, termed the \textit{essential dimension}.
We use the essential dimension to give an asymptotic expression for the number of $k$-flats spanned by $P$, for $1 \leq k \leq d-1$.
\end{abstract}

\section{Introduction}\label{sec:introduction}

Let $P$ be a set of $n$ points in a real or complex, finite-dimensional, affine space.
We say that $P$ spans a $k$-flat\footnote{We refer to affine or projective subspaces as ``flats".} $\Gamma$ if $\Gamma$ contains $k+1$ affinely independent points of $P$.
Denote the number of $k$-flats spanned by $P$ by $f_k$; in particular, $f_{-1} = 1$ and $f_0 = n$.

Our question is:

{\centering When does $P$ span more $k$-flats than $(k-1)$-flats?

}

For $k=1$, a complete answer to this question is given by a classic theorem of de Bruijn and Erd\H{o}s \cite{bruijn1948combinatorial}.
This theorem is that either the number of lines spanned by $P$ is at least $n$, or $P$ is contained in a line; furthermore, equality is achieved only if $n-1$ points of $P$ are collinear.

It might be tempting to conjecture that $f_k \geq f_{k-1}$ unless $P$ is contained in a $k$-flat.
This is easily seen to be false for $k=2$, by considering a set of $n$ points in $\mathbb{R}^3$, of which $n/2$ are incident to each of a pair of skew lines; in this case, $f_2 = n$ and $f_1 = (n/2)^2 + 2$.
In 1986, Purdy \cite{purdy1986two} showed that either $n-1$ points of $P$ lie on a plane or the union of a pair of skew lines, or $f_2 = \Omega(f_1)$.

To answer our question in higher dimensions, we introduce a new measure of the degeneracy of a point set with respect to affine subspaces.
We say the \textit{essential dimension} of a point set $P$ is the minimum $t$ such that there exists a set $\mathcal{G}$ of flats such that
\begin{enumerate}
\item $P$ is contained in the union of the flats of $\mathcal{G}$,
\item each flat $\Gamma \in \mathcal{G}$ has dimension $\dim(\Gamma) \geq 1$, and
\item $\sum_{\Gamma \in \mathcal{G}} \dim(\Gamma) = t$.
\end{enumerate}
For example, a point set that lies in the union of two skew lines has essential dimension $2$.
For any set of points $P$, we denote the essential dimension of $P$ by $K(P)$, and we omit the argument if it is obvious from the context.

We additionally denote by $g_i$ the maximum cardinality of a subset $P' \subseteq P$ such that the essential dimension of $P'$ is at most $i$; i.e., $K(P') \leq i$.

We prove
\begin{theorem}\label{th:Purdy}
For each $k$, there is a constant $c_k$ such that the following holds.
Let $P$ be a set of $n$ points in a finite dimensional real or complex affine geometry.
\begin{enumerate}
\item If $n = g_k$ (i.e., $K(P) \leq k$), then either $f_{k-1} > f_k$, or $f_{k-1} = f_k = 0$.
\item If $n - g_k > c_k$, then $f_k > f_{k-1}$.
\end{enumerate}
\end{theorem}

This theorem is a modification of a conjecture of Purdy \cite{erdos1996extremal}.
A counterexample to Purdy's original conjecture for $k \geq 3$ was given by the author, Purdy, and Smith \cite{lund2011bichromatic}; however, this counterexample left open the possibility that some variation on the conjecture (such as Theorem \ref{th:Purdy}) could be true.

The case $k=2$ of Purdy's conjecture was: if $P$ is a set of sufficiently many points, then either $P$ can be covered by two lines, or by a plane and a point, or $P$ spans at least as many planes as lines (i.e., $f_2 \geq f_1$).
This case of the conjecture appears in well-known collections of open problems in combinatorial geometry \cite{brass2005research,croft2012unsolved}, and has remained open until now.
We give counterexamples to this conjecture in section \ref{sec:constructions}, even showing that there are arbitrarily large point sets that cannot be covered by a plane and a point or by two lines such that $f_2 < (5/6)f_1 + O(1)$.

Also in Section \ref{sec:constructions}, we investigate lower bounds on the values that may be taken by $c_k$ in Theorem \ref{th:Purdy}.
In particular, we show that, even if we restrict our attention to arbitrarily large point sets, Theorem \ref{th:Purdy} does not hold for values of $c_2$ less than $4$ or $c_3$ less than $11$, and for larger $k$ we show that $c_k$ grows at least linearly with $k$.
We further give a construction that we conjecture would show that $c_k$ must grow at least exponentially with $k$, if we could properly analyze the construction in high dimensions.

Unlike the theorem of de Bruijn and Erd\H{o}s mentioned above, Theorem \ref{th:Purdy} depends crucially on the underlying field.
For example, consider the set $P$ of all points in $\mathbb{F}_q^{d}$, where $\mathbb{F}_q$ is the finite field with $q$ elements.
The number of $(d-1)$-flats spanned by $P$ is $\Theta(q^{d})$, while the number of $(d-2)$-flats is $\Theta(q^{2(d-1)})$; however, no set of essential dimension $d-1$ contains more than $q^{d-1}$ points.

Other than the result of Purdy for the case $k=2$ mentioned above, the most relevant prior work on this question is a result of Beck \cite{beck1983lattice}, who proved that there is a constant $c_k'$ depending on $k$ such that, either $f_k = \Omega(n^{k+1})$,\footnote{Here, and throughout the paper, the constants hidden by asymptotic notation depend on $k$.} or a single hyperplane contains $c_k' n$ points.
Hence, if $P$ is a set of sufficiently many points, and no hyperplane contains more than a small, constant fraction of the points of $P$, then $f_k \geq f_{k-1}$.
Considering the example of a set $P$ of $n$ points, $n/k$ of which lie on each of $k$ skew lines spanning $\mathbb{R}^{2k-1}$, shows that $c_k'$ must be a decreasing function of $k$ in this theorem.

The second claim (for $n - g_k \geq c_k$) of Theorem \ref{th:Purdy} is a consequence of the following asymptotic expression for the number of $k$-flats spanned by $P$.

\begin{theorem}\label{th:numberOfFlats}
Let $P$ be a set of $n$ points in a finite dimensional real or complex affine geometry.

For $k < K = K(P)$,
\begin{equation}\label{eqn:mainSmallk}
f_k = \Theta \left( \prod_{i=0}^k (n-g_i) \right),\end{equation}
provided that $n-g_k \geq c_k$, for a constant $c_k$ depending only on $k$.

For $k \geq K$,
\begin{equation}\label{eqn:mainLargek}
f_k = O \left(\prod_{i=0}^{2(K-1) - k} (n-g_i) \right).\end{equation}
\end{theorem}

Claim $2$ of Theorem \ref{th:Purdy} is an immediate consequence of expression (\ref{eqn:mainSmallk}) in Theorem \ref{th:numberOfFlats}.
Theorem \ref{th:numberOfFlats} is also a substantial generalization of a conjecture made by the author, Purdy, and Smith \cite{lund2011bichromatic}.

Recently, Do \cite{do2016extending} independently found a different proof a special case of (\ref{eqn:mainSmallk}).
In particular, Do shows that if $n-g_k = \Omega(n)$ then $f_k = \Omega(n^{k+1})$, for suitable choices of the implied constants.

Theorem \ref{th:numberOfFlats} additionally implies an asymptotic version of a special case of a long-standing conjecture in matroid theory.
Rota \cite{gian1971combinatorial} conjectured that the sequence of the number of flats of each rank in any geometric lattice is unimodal, and Mason \cite{mason1972matroids} proposed the stronger conjecture that the sequence is log-concave.
We have
\begin{corollary}\label{th:unimodal}
For $k < K$ such that $n-g_k \geq c_k$,
$$f_k^2 = \Omega(f_{k-1}f_{k+1}).$$
\end{corollary}
This follows immediately from Theorem \ref{th:numberOfFlats} and the easy observation that $n-g_i \leq n-g_{i-1}$ for any $i$.
Note that Corollary \ref{th:unimodal} applies only to real or complex affine geometries, and is also weaker than Rota's conjecture due the additional assumptions on $P$ and the implied constant in the asymptotic notation.

We remark that the assumption that the underlying field is either the real or complex numbers is only used for the lower bound of Theorem \ref{th:numberOfFlats}; the proofs of claim 1 of Theorem \ref{th:Purdy} and the upper bound of Theorem \ref{th:numberOfFlats} are independent of this assumption.
Claim 2 of Theorem \ref{th:Purdy} and Corollary \ref{th:unimodal} both rely on the lower bound of Theorem \ref{th:numberOfFlats}, and hence are proved only for real and complex geometry.

\subsection{Organization of the paper}
Section \ref{sec:preliminaries} reviews basic facts of projective geometry and defines notation.
Section \ref{sec:claim1} gives the proof of claim 1 of Theorem \ref{th:Purdy}.
Section \ref{sec:upperBound} gives the proof  of the upper bound of Theorem \ref{th:numberOfFlats}.
Section \ref{sec:knownResults} reviews some well-known consequences of the Szemer\'edi-Trotter theorem.
Section \ref{sec:lowerBound} gives the proof of the lower bound of Theorem \ref{th:numberOfFlats}.
Section \ref{sec:constructions} describes several new infinite families of point sets, that disprove Purdy's conjecture for $\mathbb{R}^3$, and establish lower bounds on the values that could be assumed by the constant $c_k$ in Theorem \ref{th:Purdy}.

\subsection{Acknowledgements}

I thank George Purdy for suggesting the problem, and Abdul Basit, Zoltan Kiraly, Joe Malkevich, George Purdy, and Justin Smith for various helpful conversations and suggestions.

\section{Preliminaries}\label{sec:preliminaries}

In section \ref{sec:projection}, we review some basic facts of projective geometry, and fix the relevant notation.
In section \ref{sec:context}, we define some basic constructions, and list consisely the notation used for these constructions.

\subsection{Projection}\label{sec:projection}

It suffices to prove Theorems \ref{th:Purdy} and \ref{th:numberOfFlats} for sets of points in a finite dimensional projective geometry.
Indeed, given a set of points in an affine geometry, we can add an empty hyperplane at infinity to obtain points in a projective geometry that determine the same lattice of flats.

In this section, we fix notation and review basic facts about projective geometry that we rely on in the proofs.

We denote by $\mathbb{P}^d$ the $d$-dimensional projective geometry over either $\mathbb{R}$ or $\mathbb{C}$.
We refer to projective subspaces of $\mathbb{P}^d$ as flats.

The span of a set $X \subset \mathbb{P}^d$ is the smallest flat that contains $X$, and is denoted $\overline{X}$.
Let $\Lambda, \Gamma$ be flats of $\mathbb{P}^d$.
We denote by $\overline{\Lambda,\Gamma}$ the span of $\Lambda \cup \Gamma$.
It is a basic fact of projective geometry that
\begin{equation}\label{eqn:dimSpan} \dim(\overline{\Lambda,\Gamma})  + \dim(\Lambda \cap \Gamma) = \dim(\Lambda) + \dim(\Gamma).\end{equation}
Recall that $\dim(\emptyset) = -1$.

For a $k$-flat $\Lambda$, we define the projection from $\Lambda$ to be the map
$$\pi_{\Lambda} : \mathbb{P}^{d} \setminus \Lambda \rightarrow \mathbb{P}^{d-k-1}$$
that sends a point $p$ to the intersection of the $(k+1)$-flat $\overline{p,\Lambda}$ with an arbitrary $(d-k-1)$-flat disjoint from $\Lambda$.
With a slight abuse of notation, for any set $X \subseteq \mathbb{P}^d$, we define $\pi_\Lambda(X)$ to be the image of $X \setminus \Lambda$ under projection from $\Lambda$.

For example, let $\Lambda, \Gamma$ be flats in $\mathbb{P}^d$ of dimensions $k$ and $k'$, respectively.
Then, $\pi_{\Lambda}(\Gamma)$ is defined to be the intersection of $\overline{\Gamma, \Lambda}$ with a $(d-k-1)$-flat $\Sigma$ such that $\Lambda \cap \Sigma = \emptyset$.
Together with equation (\ref{eqn:dimSpan}), this implies that $\dim(\overline{\Sigma,\Lambda})=d$, and, since we are in $\mathbb{P}^d$, we have also that $\dim(\overline{\Sigma, \Lambda, \Gamma})=d$.
Applying (\ref{eqn:dimSpan}) two more times, we have
\begin{align*}
\dim(\Sigma \cap \overline{\Lambda,\Gamma}) &= \dim(\Sigma) + \dim(\overline{\Lambda,\Gamma}) - \dim(\overline{\Lambda,\Gamma,\Sigma}), \\
&= \dim(\Sigma) + \dim(\Lambda) + \dim(\Gamma) - \dim(\Lambda \cap \Gamma) - \dim(\overline{\Lambda,\Gamma,\Sigma}), \\
&= k' - 1 - \dim(\Lambda \cap \Gamma).
\end{align*}
In other words, the projection of a $k'$-flat through a $k$-flat in $\mathbb{P}^d$ is a $(k' - 1 -\dim(\Lambda \cap \Gamma))$-flat in $\mathbb{P}^{d-k-1}$, and so

\begin{equation}\label{eqn:dimProj}\dim(\pi_\Lambda(\Gamma))  = \dim(\Gamma) - 1 - \dim(\Gamma \cap \Lambda).\end{equation}

\subsection{Context and notation}\label{sec:context}

For the remainder of the paper, we fix a point set $P$ of size $|P| = n$ in a finite dimensional real or complex projective space.

Recall that the \textit{essential dimension} $K(Q)$ of a set $Q$ of points is the minimum $t$ such that there exists a set of flats, each of dimension $1$ or more, the union of which contains $Q$, and whose dimensions sum to $t$.
The proofs in sections \ref{sec:upperBound} and \ref{sec:lowerBound} proceed primarily by isolating maximum size subsets of $P$ having specified essential dimension.

We define $g_k(Q)$ to be the maximum size of a subset $Q' \subseteq Q$ such that $K(Q') \leq k$.
We define $\mathcal{G}_k(Q)$ as a set of flats that satisfies the following conditions:
\begin{enumerate}
\item each flat in $\mathcal{G}_k$ has dimension at least $1$,
\item $\sum_{\Gamma \in \mathcal{G}_k} \dim(\Gamma) \leq k$,
\item $|\cup_{\Gamma \in \mathcal{G}_k} \Gamma \cap Q| = g_k$,
\item $|\mathcal{G}_k| \leq |\mathcal{G}_k'|$ for any set $\mathcal{G}_k'$ that satisfies conditions 1,2, and 3.
\end{enumerate}
In other words, $\mathcal{G}_k(Q)$ is a set of flats of minimum cardinality that contains a maximum cardinality set $Q' \subset Q$  with essential dimension $K(Q') \leq k$.

We further define the following functions on any point set $Q$:

\begin{tabularx}{\textwidth}{l X}
$f_k(Q)$ & the number of $k$-flats spanned by $Q$, \\
$\mathcal{F}_k(Q)$ & the set of $k$-flats spanned by $Q$, \\
$f_k^{\sigma c}(Q)$ & for $\sigma \in \{\leq, =, \geq\}$; the number of $k$-flats spanned by $Q$ that each contain at most / exactly / at least $c$ points of $Q$, \\
$\mathcal{F}^{\sigma c}_k(Q)$ & the set of flats counted by $f_k^{\sigma c}(Q)$, \\
$\mathcal{G}(Q)$ & $\mathcal{G}_{K(Q)}(Q)$. \\
\end{tabularx}

The argument to any one of these functions will be omitted when it is clear from the context, in which case the argument will most often be $P$.
This also applies to the projection operations described in section \ref{sec:projection}; for example, $\pi_\Gamma$ is shorthand for $\pi_\Gamma(P)$, and denotes the projection of $P$ from $\Gamma$.

Given a point set $Q$ and a set of flats $\mathcal{F}$, we define the number of incidences between $Q$ and $\mathcal{F}$ as
$$I(Q,\mathcal{F}) = |\{(p,\Gamma) \in Q \times \mathcal{F} \mid p \in \Gamma \}|.$$

\section{Claim 1 of Theorem \ref{th:Purdy}}\label{sec:claim1}

In this section, we establish claim 1 of Theorem \ref{th:Purdy}.

The results in this section are for weighted points.
In particular, we assume the existence of a function $W:P \rightarrow \mathbb{R}$ such that $W(p) \geq 1$ for all $p \in P$.

Given such a weight function on the points of $P$, we extend it to flats and define related weight functions for projections of $P$ as follows.
The weight of a flat $\Lambda$ is
$$W(\Lambda) = \sum_{p \in P \cap \Lambda} W(p).$$
The weight of a point $q \in \pi_\Gamma$ is
$$W_\Gamma(q) = \sum_{\substack{p \in P \mid  \\ \pi_{\Gamma}(p) = q}} W(p).$$
Note that, for any flat $\Gamma$, we have
$$\sum_{q \in \pi_\Gamma} W_\Gamma(q) + W(\Gamma) = \sum_{p \in P} W(p).$$

The following simple lemma shows how to rewrite the sum of a function of the weights of the flats spanned by $P$ in terms of the flats projected from each point $p \in P$.

\begin{lemma}\label{thm:fctFrewrite}
For any function $F$ and $k \geq 1$,
$$\sum_{\Lambda \in \mathcal{F}_k} F(W(\Lambda)) = \sum_{p \in P} \sum_{\Lambda \in \mathcal{F}_{k-1}(\pi_p)} \frac{W(p) \cdot F(W_p(\Lambda) + W(p))}{W_p(\Lambda) + W(p)}.$$
\end{lemma}
\begin{proof}
\begin{align*}
\sum_{\Lambda \in \mathcal{F}_k} F(W(\Lambda)) &= \sum_{\Lambda \in \mathcal{F}_k} F(W(\Lambda)) \sum_{p \in P \cap \Lambda} \frac{W(p)}{W(\Lambda)}, \\
&= \sum_{p \in P} \sum_{\Lambda \mid p \in \Lambda} \frac{W(p)F(W(\Lambda))}{W(\Lambda)}, \\
&= \sum_{p \in P} \sum_{\Lambda \in \mathcal{F}_{k-1}(\pi_p)} \frac{W(p) \cdot F(W(\Lambda) + W(p))}{W(\Lambda) + W(p)}.
\end{align*}
The last line uses the observation that the $k$-flats spanned by $P$ and incident to $p$ are in bijection with the $(k-1)$-flats spanned by $\pi_p$.
\end{proof}

The following lemma is the main claim of the section, from which claim 1 of Theorem \ref{th:Purdy} follows easily.
We write $\mathbb{R}^+$ for the set of strictly positive real numbers.

\begin{lemma}
Let $F:\mathbb{R}^+ \rightarrow \mathbb{R}^+$ be a non-increasing function.
Let $k \geq K$, with $f_k \geq 1$.
Then,
$$\sum_{\Lambda \in \mathcal{F}_k} F(W(\Lambda)) < \sum_{\Lambda \in \mathcal{F}_{k-1}} F(W(\Lambda)).$$
\end{lemma}

Note that the conclusion $f_k < f_{k-1}$ follows by taking $F$ to be the function that takes constant value $1$.

\begin{proof}
We proceed by induction on $K$.
In the base case, $P$ is a collinear set of at least $2$ points. Hence, for an arbitrary $p \in P$, we have 
$$\sum_{\Lambda \in \mathcal{F}_1} F(W(\Lambda)) = F(W(P)) \leq F(W(p)) < \sum_{q \in P} F(W(q)),$$
which establishes the claim.

Now, assume that the lemma holds for $K'<K$ and arbitrary $k$.
By Lemma \ref{thm:fctFrewrite}, we have
\begin{equation}\label{eqn:FWLamb}\sum_{\Lambda \in \mathcal{F}_j} F(W(\Lambda)) = \sum_{p \in P} \sum_{\Lambda \in \mathcal{F}_{j-1}(\pi_p)} \frac{W(p) \cdot F(W_p(\Lambda) + W(p))}{W_p(\Lambda) + W(p)},\end{equation}
for each of $j=k$ and $j=k-1$.

Clearly, $K(\pi_p) \leq K - 1$.
Indeed, let $p \in \Gamma \in \mathcal{G}$.
Then $\pi_p$ is contained in the union of $\pi_p(\Gamma)$ and $\pi_p(\Gamma')$ for $\Gamma' \in \mathcal{G} \setminus \Gamma$.
Since $\dim(\pi_p(\Gamma)) = \dim(\Gamma) - 1$, this provides a witness that $K(\pi_p) \leq K-1$.

Fix $p \in P$, and let
\begin{equation*}F_p(w) = \frac{W(p)F(w+W(p))}{w+W(p)},\end{equation*}
defined for positive $w$.
Since $F$ is positive valued and nonincreasing, and $W(p) \geq 1$, we have that $F_p$ is positive valued and nonincreasing.
Hence, the induction hypothesis implies that
\begin{equation}\label{eqn:FWind}\sum_{\Lambda \in \mathcal{F}_{k-1}(\pi_p)} F_p(W_p(\Lambda)) < \sum_{\Lambda \in \mathcal{F}_{k-2}(\pi_p)} F_p(W_p(\Lambda)).\end{equation}
Together, (\ref{eqn:FWLamb}) and (\ref{eqn:FWind}) imply the conclusion of the lemma.
\end{proof}

\section{Upper bound of Theorem \ref{th:numberOfFlats}}\label{sec:upperBound}

The main result of this section is Theorem \ref{th:upperBound}, which is the upper bound of Theorem \ref{th:numberOfFlats}.
Before proving the main result, we establish two lemmas on the set of $k$-flats spanned by $P$, for $k \geq K$.

\begin{lemma}\label{th:containedFlats}
Let $k \geq K$, and let $\Gamma \in \mathcal{F}_k$.
Then, there is a set $\mathcal{A} \subseteq \mathcal{G}$ of $|\mathcal A| = k + 1 - K$ flats such that $\Lambda \subseteq \Gamma$ for each $\Lambda \in \mathcal{A}$.
\end{lemma}
\begin{proof}
We first show that
\begin{equation}\label{eqn:dimUpperBoundOnUnion}\dim(\overline{\mathcal A}) \leq - 1 + \sum_{\Lambda \in \mathcal A} (\dim(\Lambda)+1).\end{equation}
We proceed by induction on $|\mathcal{A}|$.
In the base case, $|\mathcal{A}| = 1$ and the claim holds.
Suppose that $|\mathcal{A}| > 1$, and choose $\Lambda \in \mathcal{A}$ arbitrarily.
By equation (\ref{eqn:dimSpan}),
\begin{align*}
\dim(\overline{\mathcal{A}}) &= \dim(\Lambda) + \dim(\overline{\mathcal{A} \setminus \Lambda}) - \dim(\Lambda \cap \overline{\mathcal{A} \setminus \Lambda}), \\
&\leq \dim(\Lambda) +  \dim(\overline{\mathcal{A} \setminus \Lambda}) + 1.
\end{align*}
The claim follows by the inductive hypothesis.

Let $\mathcal{A} \subseteq \mathcal{G}$ be the set of flats in $\mathcal{G}$ that are contained by $\Gamma$.
We will show that $|\mathcal{A}| \geq k+1-K$.

Denote
$$\mathcal G_\Gamma = \{\Lambda \cap \Gamma \mid \Lambda \in \mathcal{G}\}.$$
Since each point of $P$ is contained in some flat of $\mathcal G$, we have $\Gamma = \overline{\mathcal{G}_\Gamma}$.
By (\ref{eqn:dimUpperBoundOnUnion}),
\begin{equation}\label{eqn:lemma5consequence}
k = \dim(\overline{\mathcal{G}_\Gamma}) \leq - 1 + \sum_{\Delta \in \mathcal{G}_\Gamma}(\dim(\Delta)+1). 
\end{equation}
If $\Delta$ is a flat contained in a flat $\Lambda$, then  $\dim(\Delta) + 1 - \dim(\Lambda) \leq 1$, and if $\Delta$ is properly contained in $\Lambda$, then $\dim(\Delta) + 1 - \dim(\Lambda) \leq 0$.
If $\Delta \in \mathcal{G}_\Gamma$ and $\Delta \in \mathcal G$, then $\Delta \in \mathcal{A}$.
Hence,
\begin{equation}\label{eqn:subtractK}
\sum_{\Delta \in \mathcal{G}_\Gamma}(\dim(\Delta)+1) - \sum_{\Lambda \in \mathcal{G}} \dim(\Lambda) \leq |\mathcal{A}|.
\end{equation}
Since $\sum_{\Lambda \in \mathcal{G}} \dim(\Lambda) = K$ by definition, the conclusion of the lemma follows from inequalities (\ref{eqn:lemma5consequence}) and (\ref{eqn:subtractK}).
\end{proof}

\begin{lemma}\label{th:boundingByProjection}
Suppose $k \geq K$.
Let $\mathcal{A} \subseteq \mathcal{G}$ such that $|\mathcal{A}| = k+1-K$ and $f_{k - \dim \overline{\mathcal{A}} - 1}(\pi_{\overline{\mathcal{A}}})$ is maximized.
Let $k' = k - \dim \overline{\mathcal{A}} - 1$.
Then,
$$ f_k = \Theta(f_{k'}(\pi_{\overline{\mathcal A}})).$$
\end{lemma}

\begin{proof}
Note that there is a natural bijection between flats of $\mathcal{F}_{k'}(\pi_{\overline{\mathcal A}})$ and flats of $\mathcal{F}_k$ that contain $\overline{\mathcal A}$.
In particular, if $\Gamma \in \mathcal{F}_k$, then, by (\ref{eqn:dimProj}), we have
$$\dim(\pi_{\overline{\mathcal A}}(\Gamma)) = k - 1 - \dim(\Gamma \cap \overline{\mathcal A}) = k'.$$
In addition, $\overline{\pi_{\overline{\mathcal A}}(\Gamma), \mathcal{A}} = \Gamma$, so the map that sends each flat in $\mathcal{F}_k$ to its projection from $\mathcal{A}$ is invertible.
Since the $f_k$ is at least the number of flats in $\mathcal{F}_k$ that contain $\overline{\mathcal A}$, we have
$$f_k \geq f_{k'}(\pi_{\overline{\mathcal A}}).$$

On the other hand, by Lemma \ref{th:containedFlats}, for each $k$-flat $\Gamma \in \mathcal{F}_k$, there is at least one set $\mathcal{B} \subset \mathcal{G}$ with $|\mathcal{B}| = k+1-K$ such that $\Lambda \subset \Gamma$ for each $\Lambda \in \mathcal{B}$.
Hence, we can define an injective function that maps each $\Gamma \in \mathcal{F}_k$ to an arbitrary pair $(\mathcal{B}, \Lambda)$ where $\mathcal{B}$ is a set as guaranteed by Lemma \ref{th:containedFlats} and $\Lambda \in \mathcal{F}_{k - \dim{\overline{A}} - 1}(\pi_{\overline{\mathcal B}})$  so that $\Gamma = \overline{\Lambda, \mathcal {B}}$.
Since there are at most $\binom{K}{k+1-K} < 2^K \leq 2^k$ choices for $\mathcal{B}$, and $f_{k'}(\pi_{\overline {\mathcal A}}) \geq f_{k - \dim{\overline{B}} - 1}(\pi_{\overline{\mathcal B}})$ by assumption, this shows that
$$f_k \leq 2^k f_{k'}(\pi_{\overline{ \mathcal A}}),$$
which completes the proof of the lemma.
\end{proof}

Next is the the main result of the section.

\begin{theorem}\label{th:upperBound}
For $0 \leq k \leq K-1$,
\begin{equation}\label{eqn:lowDimUB}f_k = O \left( \prod_{i=0}^k (n-g_i)\right).\end{equation}
 For $k \geq K$,
\begin{equation}\label{eqn:highDimUB}f_k = O \left( \prod_{i=0}^{2(K-1)-k} (n-g_i) \right).\end{equation}
\end{theorem}

\begin{proof}
The proof is structured as follows.
There is an outer induction on $K$.
For a fixed $K$, we first prove inequality (\ref{eqn:highDimUB}), and then use an induction on $k$ to prove inequality (\ref{eqn:lowDimUB}).

The base case $k=0$ and $K \geq 1$ is immediate, since $f_0 = n = n - g_0$ by definition.

Assume that inequalities (\ref{eqn:lowDimUB}) and (\ref{eqn:highDimUB}) hold for all $k$ when $K' <K$.

Suppose that $k \geq K$.
By Lemma \ref{th:containedFlats}, either $|\mathcal{G}| \geq k+1-K$, or $f_k = 0$.
If $f_k=0$, then we're done, so suppose that $|\mathcal{G}| \geq k + 1 - K$.

By Lemma \ref{th:boundingByProjection}, there is a set $\mathcal A \subseteq \mathcal G$ with $|\mathcal{A}| = k+1-K$ such that $f_k =  \Theta(f_{k'}(\pi_{\overline{\mathcal A}}))$, for $k'=k-\dim \overline { \mathcal A}-1$.

Before bounding $f_{k'}(\pi_{\overline{\mathcal A}})$, we first make some simple observations about $\pi_{\overline{\mathcal A}}$.
By definition, each point of $\pi_{\overline{\mathcal A}}$ is the image of one or more points that lie on flats of $\mathcal{G} \setminus \mathcal{A}$.
Since $\dim(\Lambda) \geq \dim(\pi_{\overline {\mathcal A}}(\Lambda))$ for any flat $\Lambda$, the fact that the preimage of $\pi_{\overline{\mathcal A}}$ is contained the flats of $\mathcal{G} \setminus \mathcal{A}$ implies that $$K(\pi_{\overline{\mathcal A}}) \leq \sum_{\Lambda \in \mathcal{G} \setminus \mathcal{A}} \dim(\Lambda) = K - \sum_{\Lambda \in \mathcal{A}} \dim(\Lambda).$$
Since $\sum_{\Lambda \in \mathcal{A}} \dim(\Lambda) \geq |\mathcal{A}| = k+1-K$, we have
$$K(\pi_{\overline{\mathcal A}}) \leq 2K - 1 - k.$$
In particular, $K(\pi_{\overline{\mathcal A}}) < K$, so we will be able to use the inductive hypothesis to bound $f_{k'}$.

Observe that the right sides of (\ref{eqn:lowDimUB}) and (\ref{eqn:highDimUB}) are both bounded above by
$O(\Pi_{i=0}^{K-1} (n-g_i))$.
Hence, by the inductive hypothesis, we have that
\begin{align}
\nonumber f_{k'}(\pi_{\overline{\mathcal A}}) &= O \left (\Pi_{i=0}^{K(\pi_{\overline{\mathcal A}})-1} (|\pi_{\overline{\mathcal A}}| - g_i(\pi_{\overline{\mathcal A}})) \right), \\
\label{eqn:upperBoundOnk'}&= O \left (\Pi_{i=0}^{2K-2-k} (|\pi_{\overline{\mathcal A}}| - g_i(\pi_{\overline{\mathcal A}})) \right).
\end{align}

Note that $|\pi_{\overline{\mathcal A}}|- g_i(\pi_{\overline {\mathcal A}}) \leq n - g_i$ for each $i$.
Indeed, the preimage of $\pi_{\overline{\mathcal A}} \cap \mathcal{G}_i(\pi_{\overline{\mathcal A}})$ has essential dimension at least $i$, so the preimage of $\pi_{\overline{\mathcal A}} \setminus (\pi_{\overline{\mathcal A}} \cap \mathcal{G}_i(\pi_{\overline{\mathcal A}}))$ provides a witness that $n-g_i \geq |\pi_{\overline{\mathcal A}}|- g_i(\pi_{\overline {\mathcal A}})$.

Together with (\ref{eqn:upperBoundOnk'}), this completes the proof of (\ref{eqn:highDimUB}).

Suppose now that $k \leq K-1$, and assume that inequality (\ref{eqn:lowDimUB}) holds for $K$ and $k' < k$.

 We claim that if $P_1,P_2$ is a partition of $P$, then
\begin{equation}\label{eqn:splitUpperBound}f_k \leq \sum_{i=-1}^k f_i(P_1) f_{k-i-1}(P_2).\end{equation}
To show this, we map $\mathcal{F}_k$ into $\bigcup_i (\mathcal{F}_i(P_1) \times \mathcal{F}_{k-i-1}(P_2))$.
Let $\Gamma \in \mathcal{F}_k$, let $\Gamma_1 = \overline{P_1 \cap \Gamma}$, and let $\Gamma_2 = \overline{P_2 \cap \Gamma}$.
Using equation (\ref{eqn:dimSpan}) and the fact that $\dim(\Gamma_1 \cap \Gamma_2) \geq -1$, we have
$$\dim(\Gamma_2) \geq k - \dim(\Gamma_1) - 1.$$
Let $\Gamma_2' \subseteq \Gamma_2$ be a  $(k - \dim(\Gamma_1) - 1)$-flat disjoint from $\Gamma_1$.
Note that $\overline{\Gamma_1, \Gamma_2'} = \Gamma$.
Also note that, if $\Gamma_1 = \Gamma$, then $\Gamma_2' = \emptyset$.
Map $\Gamma$ to the pair $(\Gamma_1, \Gamma_2')$.
Since $\Gamma$ is the unique $k$-flat spanned by $\Gamma_1$ and $\Gamma_2'$, the map is injective, and so inequality (\ref{eqn:splitUpperBound}) is established.

Let $P_1 = \cup_{\Gamma \in \mathcal{G}_k} (P \cap \Gamma)$, and let $P_2 = P \setminus P_1$.
By inequality (\ref{eqn:splitUpperBound}),
\begin{align} \nonumber f_k & \leq \sum_{i=-1}^k f_i(P_1) f_{k-i-1}(P_2), \\
& \leq (k+2) \max_{-1 \leq i \leq k} f_i(P_1) f_{k-i-1}(P_2). \label{eqn:fkUB} \end{align}

Since $|P_2| = n-g_k$, we have
\begin{equation}\label{eqn:uncontrolledPart}
f_{k-i-1}(P_2) \leq (n-g_k)^{k-i} \leq \prod_{j=0}^{k-i-1} (n - g_{k-j}).\end{equation}

For $i<k$, the inductive hypothesis implies
\begin{equation}\label{eqn:lowControlledPart} f_i(P_1) = O\left( \prod_{j=0}^i \left (|P_1| - g_j(P_1) \right)\right) = O\left( \prod_{j=0}^i (n-g_j)\right).\end{equation}

For $i=k$, inequality (\ref{eqn:highDimUB}) implies
\begin{equation}\label{eqn:highControlledPart} f_k(P_1) = O \left( \prod_{j=0}^{k-2} (|P_1| - g_j(P_1)) \right) = O\left( \prod_{j=0}^{k} (n - g_j)\right).\end{equation}

With an appropriate choice of the constants hidden in the asymptotic notation, this completes the proof of inequality (\ref{eqn:lowDimUB}).
\end{proof}

\section{Known results in the plane}\label{sec:knownResults}

In order to prove the lower bounds of Theorem \ref{th:numberOfFlats}, we will use two known consequences of the Szemer\'edi-Trotter theorem.

The Szemer\'edi-Trotter theorem was proved for real geometry by Szemer\'edi and Trotter \cite{szemeredi1983extremal}, and proved for complex geometry by T\'oth \cite{toth2015szemeredi}, and, using a different method, by Zahl \cite{zahl2012szemeredi}.

\begin{theorem}\label{th:ST}[Szemer\'edi-Trotter]
For any $t$,
$$f_1^{\geq t} = O(n^2/t^3 + n/t).$$
\end{theorem}

Theorem \ref{th:BeckErdos} was proved by Beck \cite{beck1983lattice} when the underlying field is the real numbers, and the idea of Beck's proof is easily adapted to use Theorem \ref{th:ST}.

\begin{theorem}[Beck]\label{th:BeckErdos}
There is a constant $c_b$ such that
\[f_1^{\leq c_b} = \Omega( n (n-g_1)).\]
\end{theorem}

\begin{proof}
Let $0 < c_1 < 1$ be a constant to fix later.
Counting pairs of points of $P$ that are on lines that contain between $c_b$ and $c_1 n$ points of $P$, we have
\begin{align*}
\sum_{t = c_b}^{c_1n} f_1^{=t} t^2 &= \sum_{t=c_b}^{c_1n} t^2 (f_1^{\geq t} - f_1^{\geq t+1}), \\
&= \sum_{t=c_b}^{c_1n} t^2 f_1^{\geq t} - \sum_{t=c_b + 1}^{c_1n + 1} (t-1)^2 f_1^{\geq t}, \\
&= O\left( \sum_{t = c_b}^{c_1n} t f_1^{\geq t} \right).
\end{align*}

Applying Theorem \ref{th:ST}, for appropriate choices of $c_b$ and $c_1$ we have
$$
O\left( \sum_{t = c_b}^{\sqrt{n}} tf_1^{\geq t} \right) = O\left(\sum_{t = c_b}^{\sqrt{n}} n^2/t^2 \right) \leq n^2/10,$$
and
$$
O\left(\sum_{t=\sqrt{n}}^{c_1 n} tf_1^{\geq t} \right) = O\left(\sum_{t=\sqrt{n}}^{c_1 n} n \right) \leq n^2/10.$$

Hence, either at least $n^2/4$ pairs of points are on lines that each contain at most $c_b$ points, or at least $n^2/4$ pairs of points are on lines that contain at least $c_1 n$ points.
In the first case, $f_1^{\leq c_b} \geq n^2/(4 c_b^2)$, and the theorem is proved.
Hence, we suppose that $g_1 > c_1 n$.

Let $\ell$ be a line incident to $g_1$ points of $P$, and let $P'$ be a set of $\min(g_1, n-g_1)$ points that are not incident to $\ell$.
Let $L$ be the set of lines that contain one point of $P \cap \ell$ and at least one point of $P'$.
Since each point of $P'$ is incident to $g_1$ lines of $L$, we have
$$\sum_{l \in L} |P' \cap l| = |P'|g_1.$$
Since each ordered pair of distinct points in $P'$ is incident to at most one line of $L$, we have
$$\sum_{l \in L} (|P' \cap l|^2 - |P' \cap l|) \leq |P'|^2 - |P'|.$$
By Cauchy-Schwarz,
$$\sum_{l \in L} |P' \cap l|^2 \geq \frac{\left(\sum_{\ell \in L} |P' \cap \ell| \right)^2}{|L|} = \frac{|P'|^2g_1^2}{|L|}.$$
Combining these and rearranging, we have
$$|L| \geq \min(|P'|g_1, g_1^2) = \Omega(n(n-g_1)).$$

It remains to show that a constant portion of the lines of $L$ each contain at most $c_b$ points of $P$.
Let $P''$ be the set of $n-g_1$ points of $P$ that are not incident to $\ell$.
Each pair of points of $P''$ is incident to at most $1$ line of $L$, hence the expected number of pairs of points of $P''$ on a randomly chosen line of $L$ is at most $\binom{(n-g_1)}{2}|L|^{-1} = O(1)$.
Markov's inequality implies that at least half of the lines of $L$ are each incident to at most twice the expected number of points of $P''$, and the conclusion of the theorem follows.
\end{proof}

Theorem \ref{th:aveDirac} is a variant of the ``weak Dirac" theorem, proved independently by Beck \cite{beck1983lattice}, and by Szemer\'edi and Trotter \cite{szemeredi1983extremal}.

\begin{theorem}[Weak Dirac] \label{th:aveDirac}
There is a constant $c_d$ such that, if $P$ does not include $c_d n$ collinear points, then there is a subset $B \subseteq P$ with $|B| = \Omega(|P|)$ such that each point in $B$ is incident to at least $\Omega(n)$ lines spanned by $P$.
\end{theorem}

\begin{proof}
By Theorem \ref{th:BeckErdos}, if no line contains $c_d n$ points of $P$, then $P$ spans $\Omega(n^2)$ lines.
Since no point is incident to more than $n$ such lines, there must be $\Omega(n)$ points each incident to $\Omega(n)$ of these lines.
\end{proof}

\section{Lower bound of Theorem \ref{th:numberOfFlats}}\label{sec:lowerBound}

In this section, we prove Theorem \ref{th:lowerBound}, which gives the lower bound of Theorem \ref{th:numberOfFlats}.

We will need the following consequence of the minimality of $\mathcal{G}_k$.

\begin{lemma}\label{th:minimalityOfG}
For arbitrary $k$, let $\mathcal{A} \subseteq \mathcal{G}_k$, with $|\mathcal{A}| \geq 2$, and let $\Lambda$ be an arbitrary flat.
Then
$$\sum_{\Gamma \in \mathcal{A}} \dim(\Gamma \cap \Lambda) < \dim(\Lambda).$$
\end{lemma}
\begin{proof}
Label the flats in $\mathcal{A}$ as $\Gamma_1, \ldots, \Gamma_{|\mathcal{A}|}$.
Let $\Lambda_i = \overline{\Gamma_1,\ldots,\Gamma_i,\Lambda}$, with $\Lambda_0 = \Lambda$.

We claim that
\begin{equation}\label{eqn:Lambdai}\dim(\Lambda_i) \leq \dim(\Lambda) - \sum_{j=1}^i \dim(\Gamma_j \cap \Lambda) + \sum_{j=1}^i \dim(\Gamma_j).\end{equation}
The proof of (\ref{eqn:Lambdai}) is by induction on $i$.
In the base case, $i=0$ and the claim is trivial.

Suppose (\ref{eqn:Lambdai}) holds for $i' < i$.
Then, applying equation (\ref{eqn:dimSpan}),
\begin{align*}
\dim(\overline{\Gamma_i, \Lambda_{i-1}}) + \dim(\Gamma_i \cap \Lambda_{i-1}) &= \dim(\Gamma_i) + \dim(\Lambda_{i-1}), \text{ so} \\
\dim(\Lambda_i) + \dim(\Gamma_i \cap \Lambda) &\leq \dim(\Gamma_i) + \dim(\Lambda_{i-1}).
\end{align*}
Inequality (\ref{eqn:Lambdai}) follows by the inductive hypothesis.

Hence,
\begin{equation}\label{eqn:overlineALambda}\dim(\overline{\mathcal{A}}) \leq \dim(\Lambda_{|\mathcal{A}|}) \leq \dim(\Lambda) + \sum_{\Gamma \in \mathcal{A}} \dim(\Gamma) - \sum_{\Gamma \in \mathcal{A}}\dim(\Gamma \cap \Lambda).\end{equation}

If we suppose that $\dim(\Lambda) \leq \sum_{\Gamma \in \mathcal{A}} \dim(\Gamma \cap \Lambda)$, then (\ref{eqn:overlineALambda}) implies that $\dim(\overline{\mathcal A}) \leq \sum_{\Gamma \in \mathcal{A}} \dim(\Gamma)$.
Hence, we can reduce the size of $\mathcal{G}_k$ by replacing $\mathcal{A}$ by $\overline{\mathcal{A}}$, which contradicts the minimality of $\mathcal{G}_k$.
\end{proof}

We use Lemma \ref{th:minimalityOfG} to control the projection of the points contained in flats of $\mathcal{G}_k$ from a point in $P$ that is not contained in a flat of $\mathcal{G}_k$.

\begin{lemma}\label{th:projectionOfA}
Let $k < K$, let $A = \cup_{\Gamma \in \mathcal{G}_k} \Gamma \cap P$, and let $p \in P \setminus A$.
Then, for $0 \leq i \leq k-1$,
\begin{align}
\label{eqn:projectA}g_i(\pi_p(A)) &\leq g_i(A) + k^2, \\
\label{eqn:Apreserved}|\pi_p(A)| &\geq |A| - k^2.
\end{align}
\end{lemma}

\begin{proof}
We first prove (\ref{eqn:projectA}).
Let $\Lambda \in \mathcal{G}_i(\pi_p(A))$, and let $\Lambda'$ be the preimage of $\Lambda$ under $\pi_p$; note that $\dim(\Lambda') = \dim(\Lambda)+1$.

Let 
$$\mathcal{L}(\Lambda) = \{\Gamma \cap \Lambda' \mid \Gamma \in \mathcal{G}_k, \dim(\Gamma \cap \Lambda') \geq 1\}.$$
Note that, since $p \notin A$, no flat in $\mathcal{L}(\Lambda)$ can contain $p$.
Hence, if $\mathcal{L}(\Lambda)$ contains a single flat $\Gamma$, then $\dim(\Gamma) < \dim(\Lambda') = \dim(\Lambda) + 1$.
On the other hand, if $|\mathcal{L}(\Lambda)| \geq 2$, then Lemma \ref{th:minimalityOfG} implies that $\sum_{\Gamma \in \mathcal{L}(\Lambda)}\dim(\Gamma) <\dim(\Lambda') = \dim(\Lambda) + 1$.
In either case, the flats of $\mathcal{L}(\Lambda)$ contain at most $g_{\dim(\Lambda)}(A)$ points of $A$.
Since $\Lambda$ is the projection of the points on flats of $\mathcal{L}(\Lambda)$ together with at most one point on each flat in $\mathcal{G}_k$ that does not intersect $\Lambda'$ in at least a line, we have that $|\Lambda \cap \pi_p(A)| \leq g_{\dim \Lambda}(A) + k$.
Note that, since $K((\mathcal{G}_i \cup \mathcal{G}_j) \cap P) \leq i+j$, we have that $g_i + g_j \leq g_{i+j}$ for any $i,j$.
In particular, $$\sum_{\Lambda \in \mathcal{G}_i(\pi_p(A))} g_{\dim \Lambda}(A) \leq g_i(A).$$
Hence, we have
$$g_i(\pi_p(A)) = \sum_{\Lambda \in \mathcal{G}_i(\pi_p(A))} |\Lambda \cap \pi_p(A)| \leq g_i(A) + ik,$$
which completes the proof of (\ref{eqn:projectA}).

It remains to prove (\ref{eqn:Apreserved}).
Let $\Gamma, \Gamma' \in \mathcal{G}_k$.
Since $\Gamma \cap \Gamma' = \emptyset$, we have $\dim(\Gamma' \cap \overline{\Gamma,p}) \leq 0$.
Hence, for each such pair of flats $\Gamma, \Gamma' \in \mathcal{G}_k$, there is at most one pair $q \in \Gamma, q' \in \Gamma'$ of points such that $\pi_p(q) = \pi_p(q')$.
In addition, each line incident to $p$ intersects each flat of $\mathcal{G}_k$ in at most one point, since otherwise $p$ would be contained in that flat.
Hence, the number of pairs of points $q,q' \in A$ such that $\pi_p(q) = \pi_p(q')$ is at most the number of pairs of flats in $\mathcal{G}_k$, which proves (\ref{eqn:Apreserved}).
\end{proof}

We now proceed to the main result of the section.
Theorem \ref{th:lowerBound} is slightly stronger than the lower bound of Theorem \ref{th:numberOfFlats}, to facilitate its inductive proof.

\begin{theorem}\label{th:lowerBound}
For $0 \leq k < K$, there are constants $c_l, c_k$ such that
\[f_k^{\leq c_l} = \Omega \left( \prod_{i=0}^k (n-g_i) \right ),\]
provided that $n-g_k \geq c_k$.
\end{theorem}

\begin{proof}
The proof is by induction on $k$.
The case $k=1$ is Theorem \ref{th:BeckErdos}.

Let
\begin{align*}
A &= \bigcup_{\Gamma \in \mathcal{G}_k} \Gamma \cap P, \\
B &= P \setminus A.
\end{align*}
Note that $|A| = g_k$ and $|B| = n-g_k \geq c_k$.

Let $c_1 < 1$ be a strictly positive constant to fix later.
Let $k'$ be the least integer such that $|A| - g_{k'}(A) < c_1|B| = c_1(n-g_k)$.

If $k' < k$, then no line contains $c_1|B|$ points of $B$.
Indeed, if $\ell$ is such a line, then $\mathcal{G}_{k'} \cup \ell$ contains $g_{k'} + c_1|B| > g_k$ points of $P$, which is a contradiction, since the sum of the dimensions of the flats of $\mathcal{G}_{k'} \cup \ell$ is $k' + 1 \leq k$.

If $k' = k$, let $B' = B$.
Otherwise, by Theorem \ref{th:aveDirac} (assuming $c_1 < c_d$), there is a set $B' \subseteq B$ with $|B'| = \Omega(|B|)$ such that each point of $B'$ is incident to $\Omega(|B|)$ lines spanned by $B$.

Fix $p \in B'$ arbitrarily.

We claim that, for $0 \leq i \leq k-1$,
\begin{equation}\label{eqn:projShrink}
|\pi_p| - g_i(\pi_p) = \Omega(n-g_i).
\end{equation}

Recall that $0 \leq i \leq k-1$, and $k' \leq k$, and hence, it will suffice to consider the cases that $i < k'$ and $k' \leq i \leq k-1$.

First, suppose that $i < k'$.
Since $k'$ is the least integer such that $|A| -g _{k'}(A) < c_1(n-g_k)$, we have that $|A| - g_i(A) \geq c_1(n-g_k)$.
Using this fact, together Lemma \ref{th:projectionOfA}, we have
\begin{align}
\nonumber n - g_i &= n - g_k + g_k - g_i ,\\ 
\nonumber & \leq  (c_1^{-1} + 1) (|A| - g_i(A)), \\
\nonumber &\leq (c_1^{-1} + 1)(|\pi_p(A)| - g_i(\pi_p(A)) + 2k^2), \\
\label{eq:piA-gi}& = O(|\pi_p(A)| - g_i(\pi_p(A))).
\end{align}

In the last line of the above derivation, we require $|\pi_p(A)| - g_i(\pi_p(A)) > 0$. 
This holds if $|A|-g_i(A) > 2k^2$, which holds if $c_1 c_k > 2k^2$.
Hence, we require $c_1 c_k > 2k^2$.

Since $\pi_p(A)$ is a subset of $\pi_p$, we have
$$|\pi_p| - g_i(\pi_p) \geq |\pi_p(A)| - g_i(\pi_p(A)).$$

Combined with (\ref{eq:piA-gi}), this is inequality (\ref{eqn:projShrink}).

Now, suppose that $k' \leq i \leq k-1$.

Let $\Gamma \in \mathcal{G}_{k'}(A)$.
Note that $|\overline{p,\Gamma} \cap B| < c_1 |B|$.
If this were not the case, then $\overline{p,\Gamma} \cup \mathcal{G}_{k'} \setminus \Gamma$ would have total dimension $k' + 1 \leq k$, and would contain at least $g_{k'}(A) + c_1|B| > |A| = g_k$ points.
Since $\mathcal G_{k'}$ contains at most $k'\leq k-1$ distinct flats, and the remaining points of $A$ contribute at most $|A| - g_{k'}(A) < c_1|B|$ points to $|\pi_p(A) \cap \pi_p(B)|$, we have that $|\pi_p(A) \cap \pi_p(B)| \leq kc_1|B|$.
Hence,
\begin{equation}\label{eqn:ABMostlyDisjoint}|\pi_p| \geq |\pi_p(A)| + |\pi_p(B)| - kc_1|B|.\end{equation}

Note that $g_i(\pi_p) \leq g_{i+1} \leq g_k$.
Hence, by inequality (\ref{eqn:Apreserved}) of Lemma \ref{th:projectionOfA}, we have that $|\pi_p(A)| - g_i(\pi_p) \geq |\pi_p(A)| - g_{k} \geq -O(1)$.
Combining this with inequality (\ref{eqn:ABMostlyDisjoint}) and the assumption that $|B| > c_k$, we have
\begin{align*}
|\pi_p| - g_i(\pi_p) &\geq |\pi_p(A)| + |\pi_p(B)| - k c_1 |B|- g_i(\pi_p), \\
&\geq c_d|B| - O(1) - k c_1 |B|, \\
&= \Omega( |B|),
\end{align*}
for appropriate choices of $c_1, c_k$.
Since $i\geq k'$, we have that $|B| = \Omega(n - g_i)$, and hence, this finishes the proof of inequality (\ref{eqn:projShrink}).

The inductive hypothesis applied to $\pi_p$, along with (\ref{eqn:projShrink}), implies that
\begin{equation}\label{eqn:projLB}
f_{k-1}^{\leq O(1)}(\pi_p) = \Omega\left(\prod_{i=0}^{k-1}(n-g_i)\right).
\end{equation}
Hence, each point in $B'$ is incident to $\Omega\left(\prod_{i=0}^{k-1}(n-g_i)\right)$ flats of dimension $k$ that are spanned by $P$.
Since the preimage of a point $q \in \pi_p$ may include many points of $P$, it remains to show that a substantial portion of these flats each contain at most $c_l(k)$ points of $P$.

Let $c_2$ be a large constant, to be fixed later.
Let $C \subset \pi_p$ be the set of points in $\pi_p$ such that each point in $C$ is the image of at least $c_2$ points in $P$ under projection from $p$.
Since each line incident to $p$ is incident to at most one point on each flat $\Gamma \in \mathcal{G}_k$, each point of $\pi_p(A)$ has multiplicity at most $k<c_2$.
Hence, $|C| \leq c_2^{-1}|B| = c_2^{-1}(n-g_k)$.

Let $q \in C$.
By Theorem \ref{th:upperBound},
\begin{equation}\label{eqn:incidencesWithQ}
f_{k-2}(\pi_{\overline{q,p}}) = O \left( \prod_{i=0}^{k-2} (n - g_i) \right),\end{equation}
and this is an upper bound on the number of incidences between $q$ and $(k-1)$-flats spanned by $\pi_p$.

The total number of $(k-1)$-flats spanned by $\pi_p$ that are incident to some point in $C$ is bounded above by the number of incidences between points in $C$ and flats in $\mathcal{F}_{k-1}(\pi_p)$.
Summing expression (\ref{eqn:incidencesWithQ}) over the points of $C$, and using the fact that $n-g_{k} < n-g_{k-1}$, the number of these incidences is
\begin{equation}\label{eqn:incidencesWithC}
I(C,\mathcal{F}_{k-1}(\pi_p)) = O\left(c_2^{-1} \prod_{i=0}^{k-1} (n-g_i) \right).
\end{equation}

By setting $c_2$ to be sufficiently large, we can ensure that the right side of (\ref{eqn:incidencesWithC}) is smaller than the right side of (\ref{eqn:projLB}).
Hence, we can subtract from the right side of (\ref{eqn:projLB}) the number of $k-1$ flats spanned by $\pi_p$ that contain a point of $C$ to obtain
\begin{equation}I(p, \mathcal{F}_k^{\leq c_2c_{l}(k-1)}) = \Omega\left( \prod_{i=0}^{k-1} (n-g_i) \right).\end{equation}
This bound applies for each of the $\Omega(n-g_k)$ points in $B'$, and hence (setting $c_l(k) = c_2 c_l(k-1)$)
\begin{equation}I(B', \mathcal{F}_k^{\leq c_l}) = \Omega\left( \prod_{i=0}^{k} (n-g_i) \right).\end{equation}
Since each of the flats of $\mathcal{F}_k^{\leq c_l}$ accounts for at most $c_l$ of these incidences, dividing the right side by $c_l$ immediately gives the claimed lower bound on $f_k^{\leq c_l}$.
\end{proof}

\section{Constructions}\label{sec:constructions}

In this section, we give several constructions that give lower bounds on the possible values that could be taken by $c_k$ in Theorem \ref{th:Purdy}.
We are in fact interested primarily in infinite families of examples for each $k$.
Hence, for this section, we define $c_k$ to be a function of $k$ as follows.
\begin{definition}
The constant $c_k$ is the minimum $t$ such that the following holds for all sufficiently large $n$.
If $P$ is a set of $n$ points in $\mathbb{R}^d$ or $\mathbb{C}^d$, then either
\begin{enumerate}
\item $n - g_k \leq  t$, or
\item $f_k > f_{k-1}$.
\end{enumerate}
\end{definition}
Note that this definition includes the hypothesis that $n$ is sufficiently large, which is absent in Theorem \ref{th:Purdy}.
Because of this aditional hypothesis, in order to show lower bounds of the form $c_k \geq t$, we find infinite families of point sets $S_n$, such that for each $S_n$ we have $|S_n| = n$, $f_k(S_n) \leq f_{k-1}(S_n)$, and $n-g_k(S_n) = t$.

To summarize the results on $c_k$ in this section, we show that $c_k$ increases monotonically (subsection \ref{sec:basicConstruction}), that $c_k \geq k - O(1)$ (subsection \ref{sec:highDimConstruction}), and that $c_2 \geq 4$ and $c_3 \geq 11$ (subsection \ref{sec:lowDimConstructions}).
Also in subsection \ref{sec:lowDimConstructions}, we give strong counterexamples to the conjecture of Purdy mentioned in the introduction.

In subsection \ref{sec:highDimConstruction}, we present a construction that we conjecture would show that $c_k \geq 2^{k-1}$ if it were successfully analyzed, but are unable to fully analyze the construction in higher dimensions.

All of the constructions in this section are based on the same basic idea, presented in subsection \ref{sec:basicConstruction}.

\subsection{Basic construction, and monotonicity}\label{sec:basicConstruction}

All of the constructions described in this section follow the same basic plan.
We start with a finite set $S$ of points having some known properties, then carefully select an origin point, and place a line $L$, containing a large number of points of $P$, perpendicular to the hyperplane containing $S$ and incident to the selected origin point.
This construction, along with its key properties, is described in Lemma \ref{th:raiseDimension}.

\begin{lemma}\label{th:raiseDimension}
Let $S$ be a set of $n$ points in $\mathbb{R}^d$, all contained in the hyperplane $H_0$ defined by $x_1 = 0$.
Denote by $f_k^o(S)$ the number of $k$-flats spanned by $S$ that are incident to the origin, and by $f_k^{\overline{o}}(S)$ the number of $k$-flats spanned by $S$ that are not incident to the origin; we define $f_0^o = 0$.
Let $L$ be a set of $m \geq 2$ collinear points contained in the line $\ell_0$ defined by the equations $x_i = 0$ for $i \neq 1$, and stipulate that the origin is not included in $L$.
Let $P = S \cup L$.
Then, for each $0 < k<d$,
\begin{equation}\label{eqn:const1}f_k(P) = m f_{k-1}^{\overline{o}}(S) + f_{k-1}^o(S) + f_k(S) + f_k(L).\end{equation}
\end{lemma}

\begin{proof}
Let $\Gamma \in \mathcal{G}_k(P)$.
If $\Gamma$ contains the origin and another point in $\ell_0$, then $\Gamma$ contains $\ell_0$ and hence contains each point of $L$.
In this case, $\dim(\Gamma \cap H_0) = k-1$, and there are $f_{k-1}^o(S)$ such flats spanned by $S$.
If $\Gamma$ contains exactly one point of $L$, then $\Gamma$ does not contain the origin, and $\dim(\Gamma \cap H_0) = k-1$.
Since there are $m$ choices for the point in $L$, the number of such flats is $m f_{k-1}^{\overline{o}}(S)$.
We also have those $k$-flats that are spanned individually by $S$ or $L$.
\end{proof}

Given an example that shows that $c_k \geq t$ for some $t$, Lemma \ref{th:raiseDimension} can be used to create an equally strong example for $c_{k+1}$, which implies that the sequence $c_2, c_3, \ldots$ is monotonic.

\begin{corollary}\label{th:ckMonotonic}
The sequence $c_2, c_3, \ldots$ increases monotonically.
\end{corollary}
\begin{proof}
Let $1 < k < d$ and $c \geq 1$, and let $S$ be a set of points in $\mathbb{R}^d$, such that $f_k(S) < f_{k-1}(S)$, and such that $|S| - g_k(S) = c_k$.
Embed $S$ in the hyperplane defined by $x_1 = 0$ in $\mathbb{R}^{d+1}$, so that no flat spanned by $S$ is incident to the origin.
Let $L$ be a set of $m$ points contained in the line $x_i=0$ for $i \neq 1$, and not including the origin.
Then, by Lemma \ref{th:raiseDimension}, we have
$$f_{k+1}(P) = m f_k(S) + f_{k+1}(S) < m f_{k-1}(S) + f_k(S) = f_k(P),$$
for $m$ sufficiently large.

In addition, since $|L|$ is much larger than $|S|$, we may assume that $\mathcal{G}_{k+1}(P)$ contains $L$.
Since the origin is generic relative to the flats spanned by $S$, the number of points of $S$ in a $j+1$ flat that contains the origin is bounded by the number of points in a $j$ flat.
Hence, $\mathcal{G}_{k+1}(P)$ is the union of the line that contains $L$ and $\mathcal{G}_k(S)$, and hence $|P|-g_{k+1}(P) = |S| - g_k(S) = c_k \leq c_{k+1}$.
\end{proof}

\subsection{Constructions for arbitrary dimensions}\label{sec:highDimConstruction}

We describe two constructions that work for any sufficiently large $k$.
The first uses a hypercube as the set $S$ in the construction of Lemma \ref{th:raiseDimension}, and the second uses a cross-polytope as $S$.
We are unable to fully analyze the hypercube example in arbitrary dimensions, but conjecture that a complete analysis would show that $c_k \geq 2^{k-1}$.
The cross-polytope example shows that $c_k \geq k - O(1)$.

\vspace{4mm}
\noindent \textbf{Hypercube construction.}
We use Lemma \ref{th:raiseDimension} to describe an infinite family of sets of points, with an infinite number of members for each $k \geq 2$.
In particular, $S_n^k$, for $n\geq 2^{k+1}$, is a set of $n$ points in $\mathbb{R}^{k+1}$ such that $n - g_k(S_n^k) = 2^{k-1}$.
We conjecture that $f_k(S^k_n) < f_{k-1}(S_n^{k})$ for all $k$.
Proving this conjecture would show that $c_k \geq 2^{k-1}$.
Analyzing the construction for large $k$ is related to (though possibly easier than) the open problem of characterizing the set of flats spanned by the vertices of the hypercube $[-1,+1]^d$ in $\mathbb{R}^d$ (see \cite{aichholzer1996classifying}).
It is easy, though tedious, to analyze the construction in low dimensions; however, different, specific constructions for $k=2,3$ give better bounds on $c_k$ for $k \leq 4$.

Let $S_n^k = C^k \cup L$, where $C^k = (0, \pm 1, \ldots, \pm 1)$ is the set of vertices of a $k$-dimensional hypercube, and $L$ is the set of $m = n - 2^k$ collinear points with coordinates $(i, 0, \ldots, 0)$ for $i \in [1, n-2^k]$.

We claim that $g_k(S_n^k) = m+2^{k-1}$.
That $g_k \geq m + 2^{k-1}$ follows by considering the union of $L$ and a $(k-1)$-dimensional face of $C^k$.
To show that $g_k \leq m + 2^{k-1}$, we show that $g_{k-1}(C^k) = 2^{k-1}$; the claim on $g_k(S_k)$ follows as an immediate consequence, since $\mathcal{G}_k$ must contain $L$.

We show by induction that the intersection of a $j$-flat with $C^k$ contains at most $2^j$ points, for any $j \leq k$.
Note that $C^k = C_{-1}^{k-1} \cup C_1^{k-1}$, where $C_{i}^{k-1}$ (for $i \in \{-1,1\}$) is the set of vertices of a $(k-1)$-dimensional hypercube in the $(d-2)$-flat $H_{i}$ defined by $x_{0} = i$.
Let $\Gamma$ be a flat of dimension $\dim(\Gamma) = j$.
Either $\Gamma$ is contained in $H_{-1}$, or is contained in $H_1$, or intersects each of $H_{-1}$ and $H_1$ in a $(j-1)$-flat.
Assuming the inductive hypothesis that the intersection of a $j'$-flat with $C^{k-1}$ contains at most $2^{j'}$ points, it follows that $\Gamma$ contains at most $2^{j}$ points of $C^k$.
Since the sum of the dimensions of flats in $\mathcal{G}_{k-1}(C^{k})$ is $k$, it follows that $g_{k-1}(C^{k}) \leq 2^{k-1}$.

For $k=2$ and $k=3$, an exhaustive enumeration of the flats spanned by $C^{k}$ is easy to perform by hand, and, for $k=3$, yields
\begin{align*}
f_1^o(C^3) &= 4,\\
f_1^{\overline{o}}(C^3) &= 24,\\
f_2^o(C^3) &= 6,\\
f_2^{\overline{o}}(C^3) &= 14.
\end{align*}

Together with a similar count for $k=2$, and an application of Lemma \ref{th:raiseDimension}, we have
\begin{align*}
f_1(S_2) &= 4m + 7, \\
f_2(S_2) &= 4m + 3, \\
f_2(S_3) &= 24m + 24, \\
f_3(S_3) &= 14m + 7.
\end{align*}
Hence, our conjecture holds for these cases.

\vspace{4mm}
\noindent \textbf{Cross-polytope construction.}
We describe a family of sets $T_n^j$ of points for $j \geq 2$ and $n$ sufficiently large.
The set $T_n^j$ is a set of $n=m+6j$ points in $\mathbb{R}^{3j+1}$ such that, assuming $m$ is sufficiently large, then $f_{2j+2} < f_{2j+1} < f_{2j}$.
Furthermore, $n-g_{2j+2} = 2j-2$ and $n-g_{2j+1} = 2j$.
Taking $k=2j+2$ in this construction shows that $c_k \geq k-4$ for even $k \geq 6$, and taking $k=2j+1$ shows that $c_k \geq k-1$ for odd $k \geq 5$.

Let $D = D^{3j}$ be the vertices of a $3j$-dimensional cross-polytope in $\mathbb{R}^{3j+1}$, centered at the origin, contained in the hyperplane $x_1 = 0$.
In particular, the $6j$ vertices of $D$ are of the form $(0, \ldots,0, \pm 1, 0, \ldots, 0)$, where the nonzero entries occur for some vertex in all but the first coordinate.
We use $D$ as the set $S$ in the construction of Lemma \ref{th:raiseDimension}, so $T_n^j = D \cup L$, where $L$ is a set of $m$ points in the line $x_i=0$ for $i\neq 1$.
We will assume that $m$ is large relative to $6j$.

We first show that $f_{2j+2} < f_{2j+1} < f_{2j}$.
Let $v \in D$.
If a flat $\Gamma$ contains $v$ and $-v$, then $\Gamma$ contains the origin.
Hence, the $i$-flats spanned by $D$ that don't contain the origin each contain at most one of $v,-v$.
Since the non-opposite vertices of $D$ are linearly independent, an $i$-flat contains at most $i+1$ of them, and so $f^{\overline{o}}_i(D)$ is equal to the number of ways to choose $i+1$ non-opposite vertices from $D$, which is $2^{i+1}\binom{3j}{i+1}$.
Hence, we have
$$f_i^{\overline{o}}(D) = 2^{i+1}\binom{3j}{i+1} = 2^i \binom{3j}{i} \cdot 2\frac{3j-i}{i+1} = f_{i-1}^{\overline{o}}(D) \cdot 2\frac{3j-i}{i+1}.$$
Hence, if $(3j-i)/(i+1) < 1/2$, then $f_i^{\overline{o}}(D) < f_{i-1}^{\overline{o}}(D)$.
This holds if $i \geq 2j$.
Applying Lemma \ref{th:raiseDimension}, and using the assumption that $m$ is sufficiently large, we have
\begin{align*}
f_{2j+2} = f^{\overline{o}}_{2j+1}(D)m + O(1) &< f^{\overline{o}}_{2j}(D) m + O(1) = f_{2j+1}, \\
f_{2j+1} = f^{\overline{o}}_{2j}(D) m + O(1) &< f^{\overline{o}}_{2j-1}(D) m + O(1) = f_{2j}.
\end{align*}

Now we show that $n-g_{2j+2} = 2j - 2$ and $n-g_{2j+1} = 2j$.
In particular, we show that $g_i(D) = 2i$; since $m$ is large, $g_{i+1} = m + g_i(D)$, and so $n-g_{i+1} = 6j - g_i(D) = 6j-2i$, from which the claims easily follow.

Let $\Gamma$ be an $i$-flat, for $i\geq 1$.
If $\Gamma$ contains the origin, then it is a linear subspace and hence contains at most $i$ linearly independent vectors, and hence at most $2i$ vertices of $D$.
If $\Gamma$ does not contain the origin, then it contains at most $i+1$ linearly independent vectors, and does not contain any pair $v,-v \in D$; in this case, $\Gamma$ contains at most $i+1$ vertices.
In either case, $\Gamma$ contains at most $2i$ vertices.
Since the sum of the dimensions of the flats in $\mathcal{G}_i(D)$ is $i$, it's clear from this that $g_i(D) = 2i$.

\subsection{Stronger constructions for $k=2,3$}\label{sec:lowDimConstructions}

Gr\"unbaum and Shephard found and catalogued simplicial arrangements of planes in real projective $3$-space \cite{grunbaum1984simplicial}.
Among these are several examples that (after taking the dual arrangement of points) give sets of points that span more lines than planes, and that are not contained in a pair of lines, or in a plane and a point.
In particular, the arrangement $A_1^3(18)$ gives a set of $18$ points, spanning $60$ planes and $74$ lines, such that no plane or pair of lines contains more than $9$ of the points.
Later, Alexanderson and Wetzel \cite{alexanderson1986simplicial} found an additional simplicial arrangement of planes.
In the projective dual, this arrangement gives a set of $21$ points, spanning $90$ planes and $98$ lines, such that no plane or pair of lines contains more than $10$ of the points.

We can apply Lemma \ref{th:raiseDimension} with Alexanderson and Wetzel's construction.
By taking a generic point as the origin, and $|L|$ sufficiently large, this construction gives $c_3 \geq 11$.

For $k=2$, the hypercube example in section \ref{sec:highDimConstruction} gives the lower bound $c_2 \geq 2$.
We now show a slightly more sophisticated construction that achieves the bound $c_2 \geq 4$.

Gr\"unbaum has produced a lovely and useful catalog of the known simplicial line arrangements in the real projective plane \cite{grunbaum2009catalogue}.
We use one of the arrangements he describes as the foundation for the construction.
In particular, the point set shown in figure \ref{fig:points} is dual to the arrangmement $A(8,1)$ in Gr\"unbaum's catalog.

\begin{figure}[ht]
\caption{Base for construction showing $c_k \geq 4$}\label{fig:points}
\includegraphics[width=\textwidth]{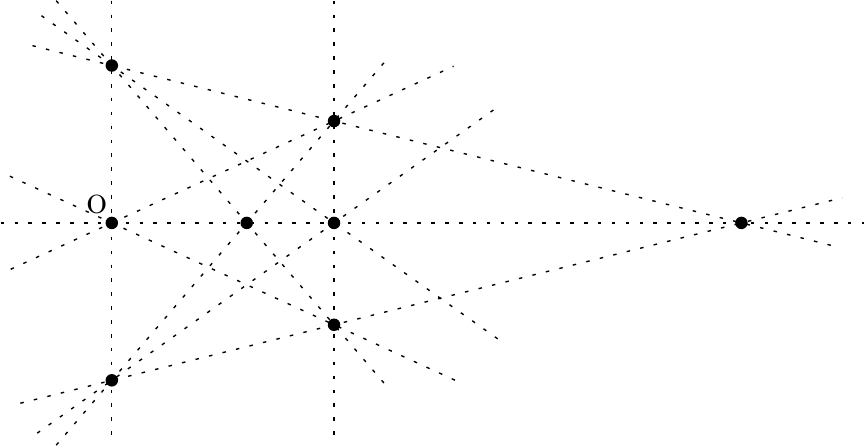}
\end{figure}

We apply Lemma \ref{th:raiseDimension} with the point set appearing in figure \ref{fig:points} as $S$, using the point marked ``o'' as the origin; i.e., let $P = S \cup L$, where $S$ is the point set in figure \ref{fig:points}, and $L$ is a set of $m$ collinear points contained in a line perpendicular to the plane spanned by $S$ and incident to the point marked ``o''.
By taking $m$ to be sufficiently large, we can ensure that the points of $L$ must be included in $\mathcal{G}_2$, and hence inspection of figure \ref{fig:points} shows that $n-g_2$ is $4$.
Further, we have  $f_1^{\overline{o}}(S) = 7$, $f_1^o(S) = 4$, and $f_0^{\overline{o}}(S) = 7$.
Hence, Lemma \ref{th:raiseDimension} gives
$$f_2(P) = 7m + 4 + 1 + 0 < 7m + 1 + 11 = f_1(P),$$
and so this construction shows that $c_2 \geq 4$.

\begin{figure}[ht]
\caption{Base for counterexample to ratio version of Purdy's conjecture}\label{fig:asymptoticExample}
\includegraphics[width=\textwidth]{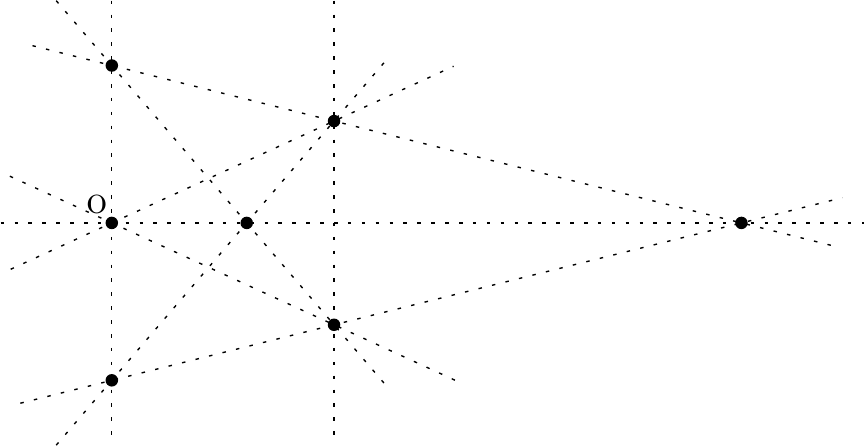}
\end{figure}

In light of the preceeding examples, it might be tempting to conjecture that, under the hypothesis of Purdy's conjecture (i.e., $P$ is a set of points that are not contained in the union of two lines or the union of a plane and a point), we at least have that $f_2 \geq f_1 - c$ for some universal constant $c$.
However, even this weaker conjecture is too optimistic.
To show this, we apply Lemma \ref{th:raiseDimension} with the point set appearing in figure \ref{fig:asymptoticExample} as $S$, using the point marked ``o" as the origin.
A brief examination of the figure reveals that $f_1^{\overline{o}}(S) = 5$ and $f_0^{\overline{o}}(S)=6$, and that $n-g_2 = 3$.
Hence, if we take $m$ to be large, it follows from Lemma \ref{th:raiseDimension} that $f_2  < (5/6)f_1 + O(1)$.

\bibliographystyle{plain}
\bibliography{purdy}

\end{document}